\documentclass[12pt]{amsart}
\usepackage{amssymb,a4wide}
\usepackage[usenames]{color} 

\usepackage[T1]{fontenc}

\newcommand{\Fo}{\mathcal{F}_{m}}

\newcommand{\C}{{\mathbb C}}

\newtheorem{Corollary}{Corollary}
\newtheorem{Theorem}{Theorem}
\newtheorem{Proposition}{Proposition}

\title[Intermediate Hankel operators on the Fock space]{Intermediate Hankel operators on the Fock space}

\author{Olivia Constantin}
\date{}

\thanks{The author was supported in part by the FWF project P30251-N35}

\address{ Olivia Constantin,
Faculty of Mathematics,
University of Vienna,
Oskar-Morgenstern-Platz 1, 
1090 Vienna, 
Austria} \email {olivia.constantin@univie.ac.at}

\begin{document}
\maketitle

\begin{abstract} 
We construct a natural sequence of  middle Hankel operators on the Fock space, i.e. operators which
are intermediate between the small and big Hankel operators. These operators are connected with 
the minimal $L^2$-norm solution operator to $\bar\partial^N$ as well as to the polyanalytic Fock spaces.
\end{abstract}

{\small

\noindent
{\it Keywords:}  middle Hankel operators; Fock spaces

\noindent
{\it AMS Mathematics Subject Classification (2000)}:  	47B35, 32A37}

\section{Introduction}
We denote by $L^2(\mu_m)$ the space of square integrable functions in $\C$ with respect to the measure $d\mu_m(z):= \frac{1}{\pi}e^{-m|z|^2}\, dA(z)$, where $m>0$ and $A$ denotes the Lebesgue area measure in $\C$.
The classical Fock space $\Fo$ is given by
$$
\Fo=\Bigl\{f \hbox{ entire} :\ \int_\C |f(z)|^2\, d\mu_m (z)<\infty\Bigr\},
$$

For an entire function $g$  such that $z^n g \in L^2(\mu_m)$ ($n \in {\mathbb N}$), the big and small 
Hankel operators with symbol $\bar{g}$ are densely defined on $\Fo$ by
$$H_{\bar{g}}f=(I-P)(\bar{g}f)\,,\quad h_{\bar{g}}f=Q(\bar{g}f)\,,$$
where $f$ is an analytic polynomial, and $P$ and $Q$ are the orthogonal projections from 
$L^2(\mu_m)$ onto $\Fo$, respectively onto $\overline{\Fo^0}=\{\bar f :\ f \in \Fo,\ f(0)=0\}$. 

Given a closed subspace $Y$ of $L^2(\mu_m)$ such that $\overline{\Fo^0} \subset Y \subset (\Fo)^\perp$, 
we define the {\it middle}, or {\it intermediate}, Hankel operator with symbol $\bar g$ by
$$H_{\bar{g}}^Yf=P_Y(\bar{g}f)\,,$$
where $f$ is an analytic polynomial, and $P_Y$ is the orthogonal projection from 
$L^2(\mu_m)$ onto $Y$. Notice that, for polynomials $f$, the following holds
$$
\|h_{\bar{g}}f\|_{L^2(\mu_m)}\le \|H_{\bar{g}}^Yf\|_{L^2(\mu_m)}\le \|H_{\bar{g}}f\|_{L^2(\mu_m)}
$$
Middle Hankel operators on Bergman spaces were considered by 
Jansson and Rochberg \cite{rochberg1}, as well as by Peng, Rochberg and Wu \cite{rochberg}, where interesting 
properties regarding Schatten classes were studied. There is one notable difference between the 
behaviour of such operators on Bergman and Fock spaces. Namely, in the Bergman setting the big 
Hankel operator is bounded if and only if the little Hankel operator is bounded,  if and only if the complex 
conjugate of the symbol belongs to the Bloch space. This makes the characterisation of the 
boundedness of middle Hankel operators on Bergman spaces trivial. On the Fock space, however, there is a big gap 
between the class of symbols generating bounded big Hankel operators (which is the space of polynomials of 
degree one) and that  generating bounded little Hankel operators (which is the space of the entire functions $f$ 
such that $|f(z)| \lesssim e^{m|z|^2/4}$). 
This makes the study of middle Hankel operators on Fock spaces 
of interest. 
In this paper we construct in a natural way a sequence of middle Hankel operators on Fock spaces, 
$( H_{\bar{g}}^{Y_N})_{N \in {\mathbb N}}$, ``increasing" in the sense that $\overline{\Fo^0}\subset ...Y_{N+1}\subset Y_{N}...\subset Y_1=(\Fo) ^\perp$, and such that  $ H_{\bar{g}}^{Y_N}$ is bounded if and only if $g$ is a polynomial 
of degree less than or equal to $N$. The operators we consider are very much related to the minimal $L^2(\mu_m)$-norm 
solution operator to the equation 
\begin{equation}\label{dbar}
(\bar\partial)^N u=f
\end{equation}
where the above holds in the sense of distributions and  $f \in \Fo$. It is easy to see that, for any polynomial $f$, the function 
$\frac{{\bar z}^N}{N!}\,f$ solves \eqref{dbar} and hence the solution of minimal norm is given by
\begin{equation}\label{sdbar}
u=(I-P_{F^{N,m}})(\tfrac{{\bar z}^N}{N!}\,f)\,,
\end{equation}
where $F^{N,m}$ is the kernel of $(\bar\partial)^N$ in $L^2(\mu_m)$, that is, the Fock space of polyanalytic functions of 
order $N$, equivalently
\begin{equation}\label{pf}
F^{N,m}=\overline{\text{span}\{\bar z^j z^k:\ 0 \le j \le N-1\,,\ k \in {\mathbb N}\}}^{L^2(\mu_m)}\,.
\end{equation}
Obviously $F^{1,m}=\Fo$. 
Polyanalytic Fock spaces and kernels appear naturally in time-frequency analysis as well as in the mathematical analysis of Landau levels
\cite{groechenig,feichtinger,hedenmalm,vasilevski}.
The connections between big Hankel operators and the minimal norm solution operator to $\bar\partial$ (i.e. the case $N=1$)
were previously exploited in \cite{rochberg2,haslinger,ortega}.

Inspired by \eqref{sdbar}, we introduce the following operator
$$\tilde H_{\bar{g}}^Nf=(I-P_{F^{N,m}})({\bar g} f)\,,$$
where $f$ is an analytic polynomial. 
For $N=1$ we recover the big Hankel operator with symbol $\bar g$, which was studied in \cite{bommier} in several complex variables. The case $N=2$ was treated in \cite{schneider} for monomial symbols.
Regarding the boundedness and compactness of $\tilde H^N_{\bar g}$  we prove the following

\begin{Theorem}\label{boundedness}  Let $g$ be an entire function such that $z^n g \in L^2(\mu_m)$  for all $n \in {\mathbb N}$. Then
 $\tilde H^N_{\bar g}$ extends to a bounded operator on $\Fo$ if and only if $g$ is a polynomial
of degree at most $N$. Moreover, $\tilde H^N_{\bar g}$ is compact if and only if  
$g$ is a polynomial
of degree strictly smaller than $N$, that is, if and only if $\tilde H^N_{\bar g}$  is the zero operator. 
\end{Theorem}
In particular, the above theorem shows that the minimal norm solution operator to $(\bar\partial)^N$ is (bounded but) not compact.

Since $\overline{\Fo^0}\not \subset (F^{N,m})^\perp$, the operator $\tilde H_{\bar{g}}^N$ is not a middle Hankel operator, however 
 it  differs  by a finite rank operator from the middle Hankel operator
 $$H_{\bar{g}}^{Y_N}(f)=(I-P_{S^{N,m}})({\bar g} f)\,,$$
where 
$$ 
S^{N,m}:= \overline{\text{span}\{\bar z^j z^k:\ 0 \le j \le N-1\,,\ k\ge j,\  k \in {\mathbb N}\}}^{L^2(\mu_m)},
$$
and $Y_N=(S^{N,m})^\perp$.
It is easy to see that 
$$\overline{\Fo^0} \subset\, ... \subset Y_{N+1}\subset Y_N\subset ...\,Y_1=
(\Fo)^\perp$$ and hence the operators $ H_{\bar{g}}^{Y_N}$ $(N\ge 1)$
are middle Hankel operators. Also, since $z^ng\in L^2(\mu_m)$ ($n\in\mathbb{N}$), the difference
$(H_{\bar{g}}^{Y_N} - \tilde H_{\bar{g}}^N)(f)=P_{span\{\bar z^j z^k:\ 0\le k< j,\ 0 \le j \le N-1\}} (\bar g f) $ has finite rank.  In particular, it follows that $ H_{\bar{g}}^{Y_N}$ is bounded (compact) if and only if $\tilde H_{\bar{g}}^N$
is bounded (compact), and hence we get

\begin{Corollary}\label{} 
 Let $g$ be an entire function such that $z^n g \in L^2(\mu_m)$  for all $n \in {\mathbb N}$. Then
$H^{Y_N}_{\bar g}$ extends to a bounded operator on $\Fo$ if and only if $g$ is a polynomial
of degree smaller than or equal to $N$. Moreover, $H^{Y_N}_{\bar g}$ is compact if and only if $g$ is a polynomial
of degree smaller than $N$. 
\end{Corollary}
\noindent Notice that, unlike $\tilde H_{\bar{g}}^N$, the operator $H^{Y_N}_{\bar g}$ is not identically zero if $g$ is a non-constant polynomial
of degree smaller than $N$.



\section{Proof of the main result}

Theorem \ref{boundedness} will follow from the characterization of the boundedness and compactness of $\tilde H^N_{\bar z^s}$ for $s\in\mathbb{N}$.
To this end, we begin by explicitly computing the projection on $F^{N,m}$ of the monomials $\bar z^s z^n$, for $s,n\in\mathbb{N}$.  
An orthonormal basis of $F^{N,m}$ (see e.g. \cite{hedenmalm}) is given by 
\begin{eqnarray*}
e_{i,r}^1(z)&:=& \sqrt{\tfrac{r!}{(r+i)!}}\,m^{(i+1)/2} z^i L_r^i(m|z|^2)\,,\qquad i \ge 0\,,\quad 0 \le r \le N-1\,,\\
e_{j,k}^2(z)&:=& \sqrt{\tfrac{j!}{(j+k)!}}\,m^{(k+1)/2} \bar{z}^k L_j^k(m|z|^2)\,,\qquad 0\le j \le N-k-1\,,\quad 1 \le k \le N-1\,,
\end{eqnarray*}
where $L_k^\alpha(x):= \sum_{i=0}^k (-1)^i {
k + \alpha \choose
k-i}
\,\frac{x^i}{i!}$
are the generalized Laguerre polynomials. For $N=1$ we obtain the standard orthonormal basis of $\Fo$:
$
e_n(z):=e_{n,0}^1(z)=\frac{m^{(n+1)/2}}{\sqrt{n!}} z^n,\quad n\ge 0.
$

\begin{Proposition}
For $s,\,n \in {\mathbb N}$ we have
$$P_{F^{N,m}}(\bar z^s z^n)= \left\{\begin{array}{l}
\displaystyle\sum_{r=0}^{N-1} \frac{r!}{(r+n-s)!}\,m^{-s} I_{n,r,(n-s)} L_r^{n-s}(m|z|^2) z^{n-s}\quad\text{if}\quad n \ge s\,,\\[0.67cm]
\displaystyle\sum_{j=0}^{N+n-s-1} \frac{j!}{(j+s-n)!}\,m^{-n} I_{s,j,(s-n)} L_j^{s-n}(m|z|^2) \bar{z}^{s-n}\quad\text{if}\quad s>n \ge s-N+1\,,\\[0.25cm]
0 \quad\text{otherwise}\,,
\end{array}\right.$$
where
$I_{a,b,c}=\int_0^\infty y^a L_b^c(y) e^{-y}\,dy\quad\text{for}\quad a,\,b,\, c \in {\mathbb N}.$
\end{Proposition}

{\it Proof.}
The result follows by using polar coordinates in
\begin{equation*}\label{1}
P_{F^{N,m}}(\bar{z}^s z^n) = \sum_{r=0}^{N-1}\sum_{i=0}^\infty \langle \bar{w}^s w^n, e_{i,r}^1\rangle e_{i,r}^1(z)
+ \displaystyle\sum_{k=1}^{N-1}\sum_{j=0}^{N-k-1} \langle \bar{w}^s w^n, e_{j,k}^2\rangle e_{j,k}^2(z)\,. \qquad\qquad\qquad\qquad\Box
\end{equation*}

\begin{Theorem}\label{boundedness1}
Let $s \in {\mathbb N}$. Then $\tilde{H}_{\bar{z}^s}^N$ extends to a bounded operator from ${\mathcal F}_{m}$ to $L^2(\mu_m)$ if and only of $s \le N$. Moreover, $\tilde{H}_{\bar{z}^s}^N$ is compact if and only if it is the zero operator.
\end{Theorem}

\begin{proof}
We first notice that $\tilde{H}_{\bar{z}^s}^N=0$ for $s<N$ since in this case $\bar{z}^sf \in F^{N,m}$ for any analytic polynomial $f$. 
We therefore assume that $s \ge N$. Using Proposition 1, a simple calculation involving polar coordinates shows that for 
any $p,\,n \in {\mathbb N}$ with $p \neq n$ we have
$
\langle \tilde{H}_{\bar{z}^s}^N e_p,\,\tilde{H}_{\bar{z}^s}^N e_n \rangle
=0\,.$
Let us now evaluate
\begin{eqnarray}
\langle \tilde{H}_{\bar{z}^s}^N e_n,\,\tilde{H}_{\bar{z}^s}^N e_n \rangle &=& \frac{m^{n+1}}{n!}\,
\Big\{\langle \bar{z}^s z^n,\bar{z}^s z^n \rangle  - \langle \bar{z}^s z^n,\,P_{F^{N,m}}(\bar{z}^s z^n) \rangle\Big\} \nonumber \\
&=& \frac{(n+s)!}{n!m^s} - \frac{m^{n+1}}{n!}\,\langle \bar{z}^s z^n,\,P_{F^{N,m}}(\bar{z}^s z^n) \rangle \,.\label{2}
\end{eqnarray}
In view of Proposition 1 we have
\begin{eqnarray}
\frac{m^{n+1}}{n!}\,\langle \bar{z}^s z^n,\,P_{F^{N,m}}(\bar{z}^s z^n) \rangle &=& 
\frac{m^{n-s+1}}{n!}\,\sum_{r=0}^{N-1} \frac{r!}{(r+n-s)!}\,I_{n,r,n-s}\, \langle \bar{z}^s z^n,\,L_r^{n-s}(m|z|^2) z^{n-s} \rangle \nonumber\\
&=& \frac{m^{-s}}{n!} \,\sum_{r=0}^{N-1} \frac{r!}{(r+n-s)!}\,(I_{n,r,n-s})^2 \,.\label{3}
\end{eqnarray}
Now
\begin{eqnarray*}
I_{n,r,n-s} &=& \int_0^\infty y^n L_r^{n-s}(y) \,e^{-y} dy  = \sum_{i=0}^r \frac{(-1)^i}{i!}\,\begin{pmatrix}
r+n-s \\
r-i
\end{pmatrix} \int_0^\infty y^{n+i}e^{-y}\,dy \\
&=&  \sum_{i=0}^r (-1)^i\,\begin{pmatrix}
r+n-s \\
r-i
\end{pmatrix}\frac{(n+i)!}{i!} = \frac{s! (r+n-s)!}{r!} \, \sum_{i=0}^r (-1)^i \begin{pmatrix}
r \\
i
\end{pmatrix} \begin{pmatrix}
n+i \\
s
\end{pmatrix} \\
&=& \frac{s! (r+n-s)!}{r!} \,  (-1)^r \begin{pmatrix}
n \\
s-r
\end{pmatrix} 
\end{eqnarray*}
by the combinatorial identity (see relation $(3.47)$ on page 27 in \cite{gould})
$$
\sum_{i=0}^r (-1)^i \begin{pmatrix}
r \\
i
\end{pmatrix} \begin{pmatrix}
n+i \\
s
\end{pmatrix}=(-1)^r \begin{pmatrix}
n \\
s-r
\end{pmatrix}. 
$$ 
Replacing $I_{n,r,n-s}$ by the above expression in relation \eqref{3} yields
\begin{eqnarray*}
\frac{m^{n+1}}{n!}\,\langle \bar{z}^s z^n,\,P_{F^{N,m}}(\bar{z}^s z^n) \rangle &=& \frac{1}{n!m^s} \, \sum_{r=0}^{N-1} 
\frac{r!}{(r+n-s)!}\,\Big[\frac{(r+n-s)!}{r!} \Big]^2 (s!)^2 \,\Big[\begin{pmatrix}
n \\
s-r
\end{pmatrix}\Big]^2 \\
&=& \frac{s!}{m^s}\,\sum_{r=0}^{N-1} \begin{pmatrix}
n \\
s-r
\end{pmatrix} \begin{pmatrix}
s \\
r
\end{pmatrix}\,.
\end{eqnarray*}
Returning to \eqref{2} we now deduce that 
$$ \Vert \tilde{H}_{\bar{z}^s}^N e_n \Vert^2 = \frac{s!}{m^s}\,\Big\{ \frac{(n+s)!}{n!s!} - \sum_{r=0}^{N-1} \begin{pmatrix}
n \\
s-r
\end{pmatrix} \begin{pmatrix}
s \\
r
\end{pmatrix}\Big\}\,.$$
In view of the combinatorial identity
$\sum\limits_{r=0}^{s} 
{n\choose 
s-r}
{s \choose
r}
={n+s\choose 
s}
$,
for $s=N$ we get
$$ \Vert \tilde{H}_{\bar{z}^s}^N e_n \Vert^2 =\frac{s!}{m^s} \quad\text{for any}\quad n \ge s\,,$$
so that $\tilde{H}_{\bar{z}^s}^N$ is bounded. If $s>N$ we obtain
\begin{eqnarray*}
\Vert \tilde{H}_{\bar{z}^s}^N e_n \Vert^2 &=& \frac{s!}{m^s} \,\Big\{ \begin{pmatrix}
n+s \\
s
\end{pmatrix} - \begin{pmatrix}
n+s \\
s
\end{pmatrix} + \sum_{r=N}^{s} \begin{pmatrix}
n \\
s-r
\end{pmatrix} \begin{pmatrix}
s \\
r
\end{pmatrix} \Big\} 
= \frac{s!}{m^s} \, \sum_{r=N}^{s} 
{n \choose
s-r} {s \choose
r
}
\end{eqnarray*}
which is a polynomial of degree $s-N$ in `n'. Thus $\tilde{H}_{\bar{z}^s}^N$ is unbounded for $s>N$. 
\end{proof}

{\it Proof of Theorem \ref{boundedness}.} The result for general symbols now follows from Theorem \ref{boundedness1} by an adaptation of the approach in Lemmas 5.3-5 from \cite{bommier}  to our context. \hfill $\Box$
\medskip


\end{document}